\documentclass[11pt,reqno]{amsart}
\usepackage[top=1.2in, bottom=1.2in, left=1.5in, right=1.5in]{geometry}

\usepackage{amsfonts, amssymb, amsmath, enumerate, hyperref}

\numberwithin{equation}{section}

\DeclareMathOperator{\E}{\mathbb{E}}

\DeclareMathOperator*{\Span}{span}
\DeclareMathOperator*{\im}{Im}
\DeclareMathOperator*{\re}{Re}
\DeclareMathOperator*{\tr}{tr}
\DeclareMathOperator*{\rank}{rank}

\renewcommand{\Pr}[2][]{\mathbb{P}_{#1} \left\{ #2 \rule{0mm}{3mm}\right\}}
\newcommand{\Prr}[2][]{\mathbb{P}_{#1} \left( #2 \rule{0mm}{3mm}\right)}

\def \C {\mathbb{C}}
\def \N {\mathbb{N}}
\def \P {\mathbb{P}}
\def \R {\mathbb{R}}

\def \EE {\mathcal{E}}
\def \BB {\mathcal{B}}
\def \LL {\mathcal{L}}

\def \VV {\mathcal{V}}

\def \a {\alpha}
\def \b {\beta}
\def \g {\gamma}
\def \e {\varepsilon}
\def \d {\delta}

\def \s {\sigma}

\def \HS {\mathrm{HS}}

\def \tran {\mathsf{T}}
\def \HS {\mathrm{HS}}
\def \< {\langle}
\def \> {\rangle}

\def \const {\mathrm{const}}

\newcommand{\norm}[1]{\left \| #1 \right \|}

\newtheorem{theorem}{Theorem}[section]
\newtheorem{proposition}[theorem]{Proposition}
\newtheorem{corollary}[theorem]{Corollary}
\newtheorem{lemma}[theorem]{Lemma}

\newtheorem{assumption}[theorem]{Assumption}

\theoremstyle{remark}
\newtheorem{remark}[theorem]{Remark}

\title[Delocalization of eigenvectors of random matrices]{Delocalization of eigenvectors of random matrices with independent entries}

\author{Mark Rudelson}
\author{Roman Vershynin}

\address{Department of Mathematics, University of Michigan, 530 Church St., Ann Arbor, MI 48109, U.S.A.}
\email{\{rudelson, romanv\}@umich.edu}
\thanks{Partially supported by NSF grants DMS 1161372, 1001829, 1265782 
and U.S. Air Force grant FA9550-14-1-0009}

\date{\today}

\subjclass[2000]{60B20}

\begin{document}

\begin{abstract}
  We prove that an $n \times n$ random matrix $G$ with independent entries
  is completely delocalized.
  Suppose the entries of $G$ have zero means, variances uniformly bounded below,
  and a uniform tail decay of exponential type.
  Then with high probability all unit eigenvectors of $G$ have all coordinates
  of magnitude $O(n^{-1/2})$, modulo logarithmic corrections.
  This comes a consequence of a new, geometric, approach to delocalization for random matrices.
\end{abstract}

\maketitle

\setcounter{tocdepth}{1}
%\tableofcontents

\section{Introduction}
%-------------

This paper establishes a complete delocalization of random matrices with independent entries haying variances of the same order of magnitude.
For an $n \times n$ matrix $G$, complete delocalization refers to the situation where
all unit eigenvectors $v$ of $G$ have all coordinates of the smallest possible magnitude
$n^{-1/2}$, up to logarithmic corrections.
For example, a random matrix $G$ with independent standard normal entries
is completely delocalized with high probability.
Indeed, by rotation invariance the unit eigenvectors $v$ are uniformly
distributed on the sphere $S^{n-1}$, so with high probability one has
$\|v\|_\infty = \max_{i \le n} |v_i| = O(\sqrt{\log(n)/n})$ for all $v$.

Rotation-invariant ensembles seem to be the only example where
delocalization can be obtained easily. Only recently was it proved
by L.~Erd\"os {\em et al.} that general symmetric and Hermitian
random matrices $H$ with independent entries are completely delocalized
\cite{ESY 1, ESY 2, ESY 3}.
These results were later extended by  L.~Erd\"os {\em et al.}  \cite{EYY} and by Tao and Vu \cite{TV eigenvectors}, see also surveys \cite{EY, E}.
Very recently, the optimal bound $O(\sqrt{\log n/n})$ was obtained by Vu and Wang \cite{VW} for the ``bulk'' eigenvectors of Hermitian matrices.
Delocalization properties with varying degrees of strength and generality were
then established for several other symmetric and Hermitian ensembles --
band matrices \cite{EN 1, EN 2, ENYY band},
sparse matrices (adjacency matrices of Erd\"os-Renyi graphs) \cite{TVW, ENYY I, ENYY II},
heavy-tailed matrices \cite{BG, BP}, and sample covariance matrices \cite{CMS}.

%In spite of a multitude of deep results and methods that were developed recently,
Despite this recent progress, no delocalization results were available
for {\em non-Hermitian random matrices} prior to the present work.
Similar to the Hermitian case, non-Hermitian random matrices
have been successful in describing various physical phenomena, see \cite{FN, FS, KA, Fisch, TWG}
and the references therein.
The distribution of eigenvectors of non-Hermitian random matrices has been studied in physics
literature, mostly focusing on correlations of certain eigenvector entries, see \cite{MC, Janik}.

\medskip

All previous approaches to delocalization in random matrices were spectral. Delocalization was obtained as a
byproduct of local limit laws, which determine eigenvalue distribution on microscopic scales.
For example, delocalization for symmetric random matrices was deduced
from a local version of Wigner's semicircle law
which controls the number of eigenvalues of $H$ falling in short intervals, even down to intervals
where the average number of eigenvalues is logarithmic in the dimension \cite{ESY 1, ESY 2, ESY 3, EYY}.

In this paper we develop a new approach to delocalization of random matrices, which is geometric
rather than spectral. The only spectral properties we rely on are crude bounds on the extreme singular values
of random matrices. As a result, the new approach can work smoothly in situations where limit spectral
laws are unknown or even impossible. In particular, one does not need to require that
the variances of all entries be the same, or even that the matrix of variances be
doubly-stochastic (as e.g. \cite{EYY}).

The main result can be stated for random variables $\xi$ with tail decay of exponential type, thus
satisfying $\Pr{|\xi| \ge t} \le 2\exp(-c t^\a)$ for some $c,\a>0$ and all $t>0$.
One can express this equivalently by the growth of moments
$\E|\xi|^p = O(p)^{p/\alpha}$ as $p \to \infty$, which is quantitatively captured by the norm
$$
\|\xi\|_{\psi_{\a}} := \sup_{p \ge 1} p^{-1/\a} (\E|\xi|^p)^{1/p} < \infty.
$$
The case $\a=2$ corresponds to {\em sub-gaussian} random variables\footnote{Standard
  properties of sub-gaussian random variables can be found in \cite[Section~5.2.3]{V RMT}.}.
It is convenient to state and prove the main result for sub-gaussian random variables,
and then deduce a similar result for general $\a>0$ using a standard truncation argument.

\begin{theorem}[Delocalization, subgaussian]				\label{thm: main}
  Let $G$ be an $n \times n$ real random matrix whose entries $G_{ij}$  are independent
  random variables satisfying $\E G_{ij} = 0$, $\E G_{ij}^2 \ge 1$ and $\|G_{ij}\|_{\psi_2} \le K$.
  Let $t \ge 2$. Then, with probability at least $1-n^{1-t}$, the matrix $G$ is completely delocalized,
  meaning that all eigenvectors $v$ of $G$ satisfy
  $$
  \|v\|_\infty \le \frac{C t^{3/2} \log^{9/2} n}{\sqrt{n}} \, \|v\|_2.
  $$
  Here $C$ depends only on $K$.
\end{theorem}

\begin{remark}[Complex matrices]				\label{rem: complex}
The same conclusion as in Theorem~\ref{thm: main} holds for a complex matrix $G$.
One just needs to require that both real and imaginary parts of all entries are independent
and satisfy the three conditions in Theorem~\ref{thm: main}. 
\end{remark}

\begin{remark}[Logarithmic losses]   \label{rem: suboptimality} 
The exponent $9/2$ of the logarithm in Theorem~\ref{thm: main} is suboptimal,
and there are several points in the proof that can be improved.
We believe that by taking care of these points, it is possible to improve the exponent to the optimal value $1/2$.
However, such improvements would come at the expense of simplicity of the argument,
while in this paper we aim at presenting the most transparent proof. The exponents $3/2$ is probably suboptimal as well.
\end{remark}

\begin{remark}[Dependence on sub-gaussian norms $\|G_{ij}\|_{\psi_2}$]
  The proof of Theorem \ref{thm: main} shows that $C$ depends polynomially
  on $K$, i.e., $C \le 2 K^{C_0}$ for some absolute constant $C_0$.
  This observation allows one to extend Theorem~\ref{thm: main} to the situation where the
  entries $G_{ij}$ of $G$ have uniformly bounded $\psi_\a$-norms, for any fixed $\a>0$.
\end{remark}

\begin{corollary}[Delocalization, general exponential tail decay]				\label{cor: psi_alpha}
  Let $G$ be an $n \times n$ real random matrix whose entires $G_{ij}$  are independent
  random variables satisfying $\E G_{ij} = 0$, $\E G_{ij}^2 \ge 1$ and $\|G_{ij}\|_{\psi_\a} \le M$.
  Let $t \ge 2$. Then, with probability at least $1-n^{1-t}$,  all eigenvectors $v$ of $G$ satisfy
  $$
  \|v\|_\infty \le \frac{C t^{\b} \log^{\g} n}{\sqrt{n}} \, \|v\|_2.
  $$
  Here $C, \b, \g$ depend only on $\a>0$ and  $M$.
\end{corollary}

Due to its spectral nature, our argument automatically establishes a stronger form of delocalization 
than was stated in Theorem~\ref{thm: main}. Indeed, one can show that not only eigenvectors
but also {\em approximate eigenvectors} of $G$ are delocalized.
Given $\e>0$, we call $v \in \R^n$ an $\e$-approximate eigenvector if there exists $z \in \C$ 
(an $\e$-approximate eigenvalue) such that $\|Gv - z v\|_2 \le \e \|v\|_2$.
We will prove the following extension of Theorem~\ref{thm: main}.

\begin{theorem}[Delocalization of approximate eigenvectors]				\label{thm: main approximate}
  Let $G$ be a random matrix as in Theorem~\ref{thm: main}, 
  and let $t \ge 2$ and $s \ge 0$.
  Then, with probability at least $1-n^{1-t(s+1)}$, 
  all $(s/\sqrt{n})$-approximate eigenvectors $v$ of $G$ satisfy
  \begin{equation}         \label{eq: main approximate}
  \|v\|_\infty \le \frac{C t^{3/2} (s+1)^{3/2} \log^{9/2} n}{\sqrt{n}} \, \|v\|_2.
  \end{equation}
  Here $C$ depends only on $K$.
\end{theorem}

\begin{remark}[Further extensions]
  The results in of this paper could be extended in several other ways. For instance,
  it is possible to drop the assumption that \emph{all} variances of the entries are of the same order and prove a similar theorem for sparse random matrices.
   One can establish the 
  \emph{isotropic delocalization} in the sense of \cite{BEKYY, KY}. We did not pursue these directions 
  since it would have made the presentation heavier. 
  It is also possible that a version of Theorems~\ref{thm: main} and \ref{thm: main approximate} 
  can be proved for {\em Hermitian matrices}. 
  We leave this direction for the future. 
\end{remark}

\subsection{Outline of the argument}
%.................

Our approach to Theorem~\ref{thm: main approximate} is based on a dimension reduction argument.
If the matrix $G$ has a localized approximate eigenvector, it will be detected from an imbalance of a suitable
projection of $G$. As we shall see, this argument yields a lower bound on $\|(G-zI)v\|_2$ that is uniform over
all unit vectors $v$ such that $\|v\|_\infty \gg 1/\sqrt{n}$. 

Let us describe this strategy in loose terms.
We fix $z$ and consider the random matrix $A = G-zI$; denote its columns by $A_j$.
Consider a projection $P$ whose kernel contains all $A_j$ except
for $j \in \{j_0\} \cup J_0$, where $j_0 \in [n]$ and $J \subset [n] \setminus \{j_0\}$
are a random uniform index and subset respectively,
with $|J_0| = l \sim \log^{10} n$.
We call such $P$ a {\em test projection}.
By its definition, triangle inequality and Cachy-Schwarz inequality, we have
for any vector $v$ that
\begin{align*}
\|Av\|_2
  &\ge \|PAv\|_2	
  = \Big\| \sum_{j=1}^n v_j PA_j \Big\|_2
  = \Big\| v_{j_0} PA_{j_0} + \sum_{j \in J_0} v_j PA_j \Big\|_2 \nonumber\\
  &\ge |v_{j_0}| \|PA_{j_0}\|_2
    - \Big( \sum_{j \in J_0} |v_j|^2 \Big)^{1/2} \Big( \sum_{j \in J_0} \|PA_j\|_2^2 \Big)^{1/2}.
\end{align*}

Suppose that $v$ is localized, say $\|v\|_\infty > l^2/\sqrt{n}$. Using the
randomness of $j_0$ and $J$, with non-negligible probability (around $1/n$)
we have $|v_{j_0}| = \|v\|_\infty > l^2/\sqrt{n}$ and
$( \sum_{j \in J_0} |v_j|^2 )^{1/2} \lesssim \sqrt{l/n}$.
On this event, we have shown that
\begin{equation}				\label{eq: Av lower intro}
\|Av\|_2 \gtrsim \frac{l^2}{\sqrt{n}} \, \|PA_{j_0}\|_2 - \frac{l}{\sqrt{n}} \, \max_{j \in J_0} \|PA_j\|_2.
\end{equation}
Since the right hand side of \eqref{eq: Av lower intro} does not depend on $v$, 
we obtained a uniform lower estimate for all localized vectors $v$.

It remains to estimate the magnitudes of $\|PA_j\|_2$ for $j \in \{j_0\} \cup J_0$.
What helps us is that the test projection $P$ can be made independent of the random
vectors $A_j$ appearing in \eqref{eq: Av lower intro}. 
Since $A=G-zI$, we can represent $A_j = G_j - z e_j$ where $e_j$ denote the
standard basis vectors. Assume first that $z$ is very close to zero, so $A_j \approx G_j$.
Then, using concentration of measure
we can argue that $\|PA_j\|_2 \approx \|PG_j\|_2 \sim \sqrt{l}$ with high probability
(and thus for all $j \in \{j_0\} \cup J_0$ simultaneously).
Substituting this into \eqref{eq: Av lower intro} we conclude that the nice lower bound
$$
\|Av\|_2 \gtrsim \frac{l^{3/2}}{\sqrt{n}}
$$
holds for all localized approximate eigenvectors $v$ corresponding 
to (approximate) eigenvalues $z$ that are very close to zero.

The challenging part of our argument is for $z$ not close to zero, namely when  
the diagonal parts $Pe_j$ dominates in the representation $PA_j = PG_j - z Pe_j$.
Estimating the magnitudes of $\|Pe_j\|_2$ might
be as difficult as the original delocalization problem.
However, it turns out that using concentration, it is possible to 
compare the terms $\|Pe_j\|_2$ with each other without knowing their magnitudes. 
This will require a careful construction of a test projection $P$.

\medskip

The rest of the paper is organized as follows.
In Section~\ref{s: preliminaries}, we recall some known linear algebraic and probabilistic facts.
In Section~\ref{s: reduction}, we rigorously develop the argument that was informally
described above. It reduces the delocalization problem to finding a test projection $P$
for which the norms of the columns $Pe_j$ have similar magnitudes.
In Section~\ref{s: distances}, we shall develop a helpful tool for estimating $\|Pe_j\|_2$, an estimate
of the distance between anisotropic random vectors and subspaces.
In Section~\ref{s: test projection}, we shall express $\|Pe_j\|_2$ in terms of such distances,
and thus will be able to compare these terms with each other.
In Section~\ref{s: proof} we deduce Theorem~\ref{thm: main approximate}.
Finally, Appendix contains auxiliary results on the smallest singular values of random matrices.

\section{Notation and preliminaries}				\label{s: preliminaries}
%-------------------

We shall work with random variables $\xi$ which satisfy the following assumption.

\begin{assumption}						\label{ass}
  $\xi$ is either real valued and satisfies
  \begin{equation}				\label{eq: mean var subgauss}
  \E \xi=0, \quad \E \xi^2 \ge 1, \quad \|\xi\|_{\psi_2} \le K,
  \end{equation}
  or $\xi$ is complex valued, where $\re \xi$ and $\im \xi$ are independent
  random variables each satisfying the three conditions in \eqref{eq: mean var subgauss}.
\end{assumption}

We will establish the conclusion of Theorem~\ref{thm: main approximate}
for random matrices $G$ with independent entries that satisfy Assumption ~\ref{ass}.
Thus we will simultaneously treat the real case and the complex case discussed in Remark~\ref{rem: complex}.

We will regard the parameter $K$ in Assumption~\ref{ass} as a constant, thus
$C, C_1, c, c_1, \ldots$ will denote positive numbers that may depend on $K$ only;
their values may change from line to line.

Without loss of generality, we can assume that $G$, as well as various other matrices
that we will encounter in the proof, have full rank.
This can be achieved by a perturbation argument, where one adds to $G$ 
an independent Gaussian random matrix $G'$ whose all entries are independent $N(0,\s^2)$ 
random variables with sufficiently small $\s>0$.  
Such perturbation will not affect the proof of Theorem \ref{thm: main approximate} since any $\e$-approximate eigenvector of $G$ will be a $(2\e)$-approximate eigenvector of $G'$ whenever $\norm{G-G'}<\e$.

By $\E_X$, $\P_X$ we denote the conditional expectation and probability with respect to a random variable $X$,
conditioned on all other variables.

The orthogonal projection onto a subspace $E$ of $\C^m$ is denoted $P_E$.
The canonical basis of $\C^n$ is denoted $e_1,\ldots,e_n$.

Let $A$ be an $m \times n$ matrix; $\|A\|$ and $\|A\|_\HS$ denote
the operator norm and Hilbert-Schmidt (Frobenius) norm of $A$, respectively.
The singular values $s_i(A)$ are the eigenvalues of $(A^*A)^{1/2}$
arranged in a non-increasing order; thus $s_1(A) \ge \cdots \ge s_r(A) \ge 0$ where $r = \min(m,n)$.
The extreme singular values have special meaning, namely
$$
s_1(A) = \|A\| = \max_{x \in S^{n-1}} \|Ax\|_2, \quad
s_m(A) = \|A^\dagger\|^{-1} = \min_{x \in S^{n-1}} \|Ax\|_2 \quad \text{(if $m \ge n$)}.
$$
Here $A^\dagger$ denotes the Moore-Penrose pseudoinverse of $A$, see e.g. \cite{GVL}.
We will need a few elementary properties of singular values.

\begin{lemma}[Smallest singular value]				\label{lem: smallest sv}
  Let $A$ be an $m \times n$ matrix and $r = \rank(A)$.
  \begin{enumerate}[(i)]
    \item Let $P$ denote the orthogonal projection in $\R^n$ onto $\im(A^*)$.
      Then $\|Ax\|_2 \ge s_r(A) \|Px\|_2$ for all $x \in \R^n$.
    \item Let $r=m$.
      Then for every $y \in \R^m$, the vector $x = A^\dagger y \in \R^n$ satisfies
      $y = Ax$ and $\|y\|_2 \ge s_m(A) \|x\|_2$. 
    \end{enumerate}
\end{lemma}

%\begin{proof}
%(i)
%Assume first that $x \in \im(A^\tran)$.
%Let $A = \sum_{i=1}^m s_i u_i v_i^\tran$ be the singular value decomposition of $A$.
%Then $\im(A^\tran) = \Span(v_i)_{i \le r}$. By orthogonality, we can express
%$x = \sum_{i=1}^r a_i v_i$ for some $a = (a_1,\ldots,a_r)$ such that $\|a\|_2 = \|x\|_2$.
%Then
%$$
%\|Ax\|_2^2 = \Big\| \sum_{i=1}^r s_i a_i u_i \Big\|_2^2
%= \sum_{i=1}^r |s_i a_i|^2 \ge |s_r|^2 \|a\|_2^2 = |s_r|^2 \|x\|_2^2,
%$$
%as claimed. For general $x$, note that $A(I-P)=0$, hence $Ax = APx$
%and the the result follows by using the first part of the argument with $Px$ instead of $x$.
%(ii)
%This result follows by choosing $x=A^\dagger y$, where $A^\dagger$ denotes the
%Moore-Penrose pseudoinverse of $A$. Specifically, if $A = \sum_{i=1}^m s_i u_i v_i^\tran$
%is the singular value decomposition of $A$, then $A^\dagger = \sum_{i=1}^m s_i^{-1} v_i u_i^\tran$.
%Then $A A^\dagger = I$ and $\|A^\dagger\| = s_m^{-1}$. The conclusion easily follows.
%\end{proof}

%\begin{lemma}[Negative second moment identity, \cite{TV ESD} Lemma A.4] \label{lem: nsmi}
%  Let $m \ge n$, and let $A$ be an $m \times n$ matrix.
%  Let $A_i$ denote the $i$-th column of $A$, and $E_i = \Span(A_j)_{j \le n, \, j \ne i}$.
%  Suppose $d(A_i, E_i) > 0$ for all $i \le n$. Then
%  $$
%  \sum_{i=1}^n \s_i(A)^{-2} = \sum_{i=1}^n d(A_i, E_i)^{-2}.
%  $$
%\end{lemma}

Appendix~\ref{app} contains estimates of the smallest singular values of random matrices.

\medskip

Next, we state a concentration property of sub-gaussian random vectors.

\begin{theorem}[Sub-gaussian concentration]		 				\label{thm: concentration}
  Let $A$ be a fixed $m \times n$ matrix.
  Consider a random vector $X = (X_1,\ldots,X_n)$ with independent components
  $X_i$ which satisfy Assumption~\ref{ass}.
  \begin{enumerate}[(i)]
    \item {\em (Concentration)} For any $t \ge 0$, we have
      $$
      \Pr{ \big| \|AX\|_2 - M \big| > t }
      \le 2 \exp \Big( - \frac{c t^2}{\|A\|^2} \Big)
      $$
      where $M = (\E\|AX\|_2^2)^{1/2}$ satisfies $\|A\|_\HS \le M \le K \|A\|_\HS$.
    \item {\em (Small ball probability)} For every $y \in \R^m$, we have
      $$
      \Pr{ \|AX-y\|_2 < \frac{1}{6} (\|A\|_\HS + \|y\|_2) }
      \le 2 \exp \Big( - \frac{c \|A\|_\HS^2}{\|A\|^2} \Big).
      $$
    \end{enumerate}
  In both parts, $c = c(K) > 0$ is polynomial in $K$. 
\end{theorem}

This result can be deduced from Hanson-Wright inequality. For part (ii), this was done
in \cite{LMOT}. A modern proof of Hanson-Wright inequality and deduction of both
parts of Theorem~\ref{thm: concentration} are discussed in \cite{RV Hanson-Wright}.
There $X_i$ were assumed to have unit variances; the general case follows
by a standard normalization step.

Sub-gaussian concentration paired with a standard covering argument yields
the following result on norms of random matrices, see \cite{RV Hanson-Wright}.

\begin{theorem}[Products of random and deterministic matrices]			\label{thm: product norm}
  Let $B$ be a fixed $m \times N$ matrix, and $G$ be an $N \times n$
  random matrix with independent entries that satisfy
  $\E G_{ij} = 0$, $\E G_{ij}^2 \ge 1$ and $\|G_{ij}\|_{\psi_2} \le K$.
  Then for any $s, t \ge 1$ we have
  $$
  \Pr{ \|BG\| > C(s \|B\|_\HS + t \sqrt{n} \|B\|) } \le 2 \exp(- s^2 r - t^2 n).
  $$
  Here $r = \|B\|_\HS^2 / \|B\|_2^2$ is the stable rank of $B$, and
  $C=C(K)$ is polynomial in $K$.
\end{theorem}

\begin{remark}					\label{rem: product norm special cases}
A couple of special cases in Theorem~\ref{thm: product norm} are worth mentioning.
If $B=P$ is a projection in $\R^N$ of rank $r$, then
$$
\Pr{ \|PG\| > C(s \sqrt{r} + t \sqrt{n}) } \le 2 \exp(- s^2 r - t^2 n).
$$
The same holds if $B=P$ is an $r \times N$ matrix such that $P P^* = I_r$.

In particular, for $B = I_N$ we obtain
$$
\Pr{ \|G\| > C(s \sqrt{N} + t \sqrt{n}) } \le 2 \exp(- s^2 N - t^2 n).
$$
\end{remark}

\section{Reducing delocalization to the existence of a test projection}		\label{s: reduction}
%--------------

We begin to develop a geometric approach to delocalization of random matrices.
The first step, which we discuss in this section, is presented for a general random matrix $A$.
Later it will be used for $A = G - z I_n$ where $G$ is the random matrix from Theorem~\ref{thm: main approximate} and
$z \in \C$.

We will first try to bound the probability of the following {\em localization event} for
a random matrix $A$ and parameters $l,W,w>0$:
\begin{equation}				\label{eq: LL}
\LL_{W,w} =
\Big\{ \exists v \in S^{n-1}:
  \|v\|_\infty > W \sqrt{\frac{l}{n}}
  \text{ and }
  \|Av\|_2 \le \frac{w}{\sqrt{n}} \Big\}.
\end{equation}
We will show that $\LL_{W,w}$ is unlikely
for $l \sim \log^2 n$, $W \sim \log^{7/2} n$ and $w=\const$.

\medskip
In this section, we reduce our task to the existence of a certain linear
map $P$ which reduces dimension from $n$ to $\sim l$,
and which we call a {\em test projection}.

To this end,
given an $m \times n$ matrix $B$, we shall denote by $B_j$ the $j$-th column of $B$,
and for a subset $J \subseteq [n]$, we denote by $B_J$ the submatrix of $B$
formed by the columns indexed by $J$.
Fix $n$ and $l \le n$, and define the set of pairs
$$
\Lambda = \Lambda(n,l) = \big\{ (j,J) :\; j \in [n], \, J \subseteq [n] \setminus \{j\}, \, |J|=l-1 \big\}.
$$
We equip $\Lambda$ with the uniform probability measure.

\begin{proposition}[Delocalization from test projection]				\label{prop: deloc from P}
  Let $l \le n$.
  Consider an $n \times n$ random matrix $A$ with an arbitrary distribution.
  Suppose that to each $(j_0, J_0) \in \Lambda$ corresponds a number $l' \le n$ and
  an $l' \times n$ matrix $P = P(n, l, A, j_0, J_0)$ with the following properties:
  \begin{enumerate}[(i)]
    \item $\|P\| \le 1$;
    \item $\ker(P) \supseteq \{A_j\}_{j \not\in \{j_0\} \cup J_0}$.
  \end{enumerate}
  Let $\a, \kappa> 0$. Let $w>0$ and $W = \frac{w}{\kappa l} + \frac{\sqrt{2}}{\a}$.
  Then we can bound the probability of the localization event \eqref{eq: LL} as follows:
  $$
  \Prr[A]{ \LL_{W,w} }
  \le 2 n \cdot \E_{(j_0,J_0)} \Prr[A]{\BB_{\a,\kappa}^c \,|\, (j_0, J_0)}
  $$
  where $\BB_{\a,\kappa}$ denotes the following {\em balancing event}:
  \begin{equation}				\label{eq: balancing}
  \BB_{\a,\kappa}
  = \left\{ \|PA_{j_0}\|_2 \ge \a \|PA_{J_0}\| \text{ and } \|PA_{j_0}\|_2 \ge \kappa \sqrt{l} \right\}.
  \end{equation}
\end{proposition}

Proposition~\ref{prop: deloc from P} states that in order to establish delocalization 
(as encoded by the complement of the event $\LL_{W,w}$), it is enough to find a test projection $P$ 
which satisfies the balancing property $\BB_{\a,\kappa}$.

\begin{proof}
Let $v \in S^{n-1}$, $(j_0, J_0) \in \Lambda$ let $P$ be as in the statement.
Using the properties (i) and (ii) of $P$, we have
\begin{align}
\|Av\|_2
  &\ge \|PAv\|_2	
  = \Big\| \sum_{j=1}^n v_j PA_j \Big\|_2
  = \Big\| v_{j_0} PA_{j_0} + \sum_{j \in J_0} v_j PA_j \Big\|_2 \nonumber\\
  &\ge |v_{j_0}| \|PA_{j_0}\|_2 - \Big( \sum_{j \in J_0} |v_j|^2 \Big)^{1/2} \|PA_{J_0}\|.		\label{eq: Av}
\end{align}
The event $\BB_{\alpha, \kappa}$ will help us balance
the norms $\|PA_{j_0}\|_2$ and $\|PA_{J_0}\|$, while the following
elementary lemma will help us balance the coefficients $v_i$.

\begin{lemma}[Balancing the coefficients of $v$]					\label{lem: balancing v}
  For a given $v \in S^{n-1}$ and for random $(j_0, J_0) \in \Lambda$, define
  the event
  $$
  \VV_v = \Big\{ |v_{j_0}| = \|v\|_\infty \text{ and } \sum_{j \in J_0} |v_j|^2 \le \frac{2l}{n} \Big\}.
  $$
  Then 
  $$
  \Prr[(j_0,J_0)]{\VV_v} \ge \frac{1}{2n}.
  $$
\end{lemma}

\begin{proof}[Proof of Lemma~\ref{lem: balancing v}]
Let $k_0 \in [n]$ denote a coordinate for which $|v_{k_0}| = \|v\|_\infty$.
Then
\begin{equation}				\label{eq: total probability}
\Prr[(j_0,J_0)]{\VV_v} \ge \Prr[(j_0,J_0)]{\VV_v \,|\, j_0=k_0} \cdot \Pr[(j_0,J_0)]{j_0=k_0}.
\end{equation}
Conditionally on $j_0=k_0$, the distribution of $J_0$ is uniform in the
set $\{ J \subseteq [n] \setminus \{k_0\}, \, |J|=l-1 \}$. Thus using
Chebyshev's inequality we obtain
\begin{align*}
\Prr[(j_0,J_0)]{\VV_v^c \,|\, j_0=k_0}
  &= \Pr[J_0]{ \sum_{j \in J_0} |v_j|^2 > \frac{2l}{n} \,\Big|\, j_0=k_0}\\
  &\le \frac{n}{2l} \E_{J_0} \left[ \sum_{j \in J_0} |v_j|^2 \,\Big|\, j_0=k_0 \right]\\
  &= \frac{n}{2l} \cdot \frac{l-1}{n} (\|v\|_2^2 - |v_{k_0}|^2)
  \le \frac{1}{2}.
\end{align*}
Moreover, $\Pr[(j_0,J_0)]{j_0=k_0} = \frac{1}{n}$. Substituting into \eqref{eq: total probability},
we complete the proof.
\end{proof}

Assume that a realization of the random matrix $A$ satisfies
\begin{equation}				\label{eq: Bal conditioned on A}
\Prr[(j_0,J_0)]{\BB_{\a,\kappa} \,|\, A} > 1- \frac{1}{2n}.
\end{equation}
(We will analyze when this event occurs later.)
Combining with the conclusion of Lemma~\ref{lem: balancing v},
we see that there exists $(j_0,J_0) \in \Lambda$ such that
both events $\VV_v$ and $\BB_{\a,\kappa}$ hold.
Then we can continue estimating $\|Av\|_2$ in \eqref{eq: Av}
using $\VV_v$ and $\BB_{\a,\kappa}$ as follows:
$$
\|Av\|_2
\ge \|v\|_\infty \|PA_{j_0}\|_2 - \sqrt{\frac{2l}{n}} \, \|PA_{J_0}\|
\ge \Big( \|v\|_\infty - \frac{1}{\a} \sqrt{\frac{2l}{n}} \Big) \kappa \sqrt{l},
$$
provided the right hand side is non-negative.
In particular, if $\|v\|_\infty > W \sqrt{l/n}$ where $W = \frac{w}{\kappa l} + \frac{\sqrt{2}}{\a}$,
then $\|Av\|_2 > w/\sqrt{n}$. Thus the localization event $\LL_{W,w}$ must fail.

Let us summarize.
We have shown that the localization event $\LL_{W,w}$
implies the failure of the event \eqref{eq: Bal conditioned on A}. The probability of
this failure can be estimated using Chebyshev's inequality and Fubini theorem as follows:
\begin{align*}
\Prr[A]{\LL_{W,w}}
  &\le \Prr[A]{ \Pr[(j_0,J_0)]{\BB_{\a,\kappa}^c \,|\, A} > \frac{1}{2n} } \\
  &\le 2n \cdot \E_A \Prr[(j_0,J_0)]{\BB_{\a,\kappa}^c \,|\, A}\\
  &= 2n \cdot  \E_{(j_0,J_0)} \Prr[A]{\BB_{\a,\kappa}^c \,|\, (j_0, J_0)}.
\end{align*}
This completes the proof of Proposition~\ref{prop: deloc from P}.
\end{proof}

\subsection{Strategy of showing that the balancing event is likely}
%............
Our goal now is to construct a test projection $P$ as in Proposition~\ref{prop: deloc from P}
in such a way that the balancing event $\BB_{\alpha, \kappa}$ is likely
for the random matrix $A = G - z I_n$ and for fixed $(j_0, J_0) \in \Lambda$ and $z \in \C$.
We will be able to do this for $\a \sim (l \log^{3/2} n)^{-1}$ and $\kappa = c$.

We might choose $P$ to be the orthogonal projection with 
$$
\ker(P) = \{A_j\}_{j \not\in \{j_0\} \cup J_0}.
$$
In reality, $P$ will be a bit more adapted to $A$.
Let us see what it will take to prove the two inequalities defining the balancing event
$\BB_{\alpha, \kappa}$ in \eqref{eq: balancing}.
The second inequality can be deduced from the small ball probability estimate,
Theorem~\ref{thm: concentration}(ii). Turning to the first inequality,
note that 
$$
\|PA_{J_0}\| \sim \max_{j \in J_0} \|PA_j\|_2
$$ 
up to a polynomial factor in $|J_0| = l-1$ (thus logarithmic in $n$).
So we need to show that 
$$
\|PA_{j_0}\|_2 \gtrsim \|PA_j\|_2 \quad \text{for all } j \in J_0.
$$

Since $A = G - zI_n$, the columns $A_i$ of $A$ can be expressed as $A_i = G_i - z e_i$.
Thus, informally speaking, our task is to show that with high probability,
\begin{equation}				\label{eq: two terms balance}
\|PG_{j_0}\|_2 \gtrsim \|PG_j\|_2, \quad \|Pe_{j_0}\|_2 \gtrsim \|Pe_j\|_2
\quad \text{for all } j \in J_0.
\end{equation}
The first inequality can be deduced from
sub-gaussian concentration, Theorem~\ref{thm: concentration}.
The second inequality in \eqref{eq: two terms balance} is challenging,
and most of the remaining work is devoted to validating it.
Instead of estimating $\|Pe_j\|_2$, we will compare these terms with each other.

Later, in Proposition~\ref{prop: Pei via distances}, we will relate $\|Pe_j\|_2$
to distances between anisotropic random vectors and subspaces.
We will now digress to develop a general bound on such distances,
which may be interesting on its own.

\section{Distances between anisotropic random vectors and subspaces}		\label{s: distances}
%--------------

We will be interested in the distribution of the distance $d(X, E_k)$ 
between a random vector $X \in \R^n$  
and a $k$-dimensional random subspace $E_k$ spanned 
by $k$ independent vectors $X_1,\ldots, X_k \in \R^n$. 
A number of arguments in random matrix theory that appeared in the recent years 
rely on controlling such distances, see e.g. 
\cite{TV pm1, TV ESD, RV square, RV rectangular, EYY bulk}.

Let us start with the {\em isotropic} case, where the random vectors in question 
have all independent coordinates. Here one can use Theorem~\ref{thm: concentration} 
to control the distances. 

\begin{proposition}[Distances between isotropic random vectors and subspaces]	\label{prop: distance}
  Let $1 \le k \le n$. Consider independent random vectors $X, X_1, X_2, \ldots, X_k$
  with independent coordinates satisfying Assumption~\ref{ass}.
  Consider the subspace $E_k = \Span(X_i)_{i=1}^k$.
  Then 
  $$
  \sqrt{n-k} \le \left( \E d(X, E_k)^2 \right)^{1/2} \le K \sqrt{n-k}.
  $$
  Furthermore, one has
  \begin{enumerate}[(i)]
    \item $\Pr{d(X, E_k) \le \frac{1}{2} \sqrt{n-k}} \le 2 \exp(-c(n-k))$;
    \item $\Pr{d(X, E_k) > 2K \sqrt{n-k}} \le 2 k \exp(-c(n-k))$.
   \end{enumerate}
  Here $c = c(K)>0$.
\end{proposition}

\begin{proof}
By adding small independent Gaussian perturbations to the vectors $X_j$, we can 
assume that $\dim(E_k) = k$ almost surely.
We can represent the distance as
$$
d(X,E_k) = \|P_{E_k^\perp} X\|_2.
$$
Since $\|P_{E_k^\perp}\| = 1$ and $\|P_{E_k^\perp}\|_\HS = \sqrt{\dim(E_k^\perp)} = \sqrt{n-k}$, 
the conclusion of the proposition follows from Theorem~\ref{thm: concentration} 
upon choosing $t = \frac{1}{2} \sqrt{n-k}$ in part (i) and $t = K \sqrt{n-k}$ in part (ii).
\end{proof}

In this paper, we will need to control the distances in the more difficult {\em anisotropic} case, 
where all random vectors are transformed by a fixed linear map $D$. 
In other words, we will be interested in distances of the form $d(DX, E_k)$ 
where $E_k$ is the span of the vectors $DX_1, \ldots, DX_k$.
An ideal estimate should look like 
\begin{equation}         \label{eq: ideal distance}
d(DX, E_k) \asymp \Big( \sum_{i>k} s_i(D)^2 \Big)^{1/2} \quad \text{with high probability},
\end{equation}
where $s_i(D)$ are the singular values of $D$ arranged in the non-increasing order. 
To see why such estimate would make sense, note that in the isotropic case where $D=I_n$ 
the distance is of order $\sqrt{n-k}$, while for $D$ of rank $k$ or lower, the distance is zero.

The following result, based again on Theorem~\ref{thm: concentration},
establishes a somewhat weaker form of \eqref{eq: ideal distance} with exponentially high probability.

\begin{theorem}[Distances between anisotropic random vectors and subspaces]	\label{thm: distance}
  Let $D$ be an $n \times n$ matrix with singular values\footnote{As usual, we arrange the singular values in a
  non-increasing order.} $s_i = s_i(D)$, and define $\bar{S}_m^2 = \sum_{i > m} s_i^2$ for $m \ge 0$.
  Let $1 \le k \le n$. Consider independent random vectors $X, X_1, X_2, \ldots, X_k$
  with independent coordinates satisfying Assumption~\ref{ass}.
  Consider the subspace
  $E_k = \Span(DX_i)_{i=1}^k$.
  Then for every $k/2 \le k_0 < k$ and $k < k_1 \le n$, one has
  \begin{enumerate}[(i)]
    \item $\Pr{d(DX, E_k) \le c \bar{S}_{k_1}} \le 2 \exp(-c (k_1-k))$;
    \item $\Pr{d(DX, E_k) > C M (\bar{S}_{k_0} + \sqrt{k} \, s_{k_0+1})}
      \le 2 k \exp(-c (k-k_0))$.
   \end{enumerate}
  Here $M = C k \sqrt{k_0} / (k-k_0)$ and $C = C(K)$, $c = c(K)>0$.
\end{theorem}

\begin{remark}
  It is important that the probability bounds in Theorem~\ref{thm: distance} 
  are exponential in $k_1-k$ and $k-k_0$.
  We will later choose $k \sim l \sim \log^2 n$ and $k_0 \approx (1-\d)k$, $k_1 \approx (1+\d)k$, where $\d \sim 1/\log n$.
  This will allow us to make the exceptional probabilities Theorem~\ref{thm: distance} smaller
  than, say, $n^{-10}$.
\end{remark}
		
\begin{remark}						\label{rem: distance sharper}
  As will be clear from the proof, one can replace the distance  $d(DX, E_k)$
  in part (ii) of Theorem~\ref{thm: distance} by the following bigger quantity:
  $$
  \inf \Big\{ \Big\| DX - \sum_{i=1}^k a_i D X_i \Big\|_2 :\;  a = (a_1,\ldots,a_k), \, \|a\|_2 \le M \Big\}.
  $$
  %We will use this fact later.
\end{remark}

\begin{proof}[Proof of Theorem \ref{thm: distance}]
(i)
We can represent the distance as
$$
d(DX,E_k) = \|BX\|_2 \quad \text{where} \quad B = P_{E_k^\perp} D.
$$
We truncate the singular values of $B$  by defining an $n \times n$
matrix $\bar{B}$ with the same left and right singular vectors as $B$, and with
singular values
$$
s_i(\bar{B}) =
\min \{ s_i(B), s_{k_1-k}(B) \}.
%= \begin{cases}
%s_{k_1-k}(B), & i \le k_1-k \\
%s_i(B), & i > k_1-k.
%\end{cases}
$$
Since $s_i(\bar{B}) \le s_i(B)$ for all $i$, we have $\bar{B} \bar{B}^* \preceq B B^*$ in the p.s.d. order,
which implies
\begin{equation}				\label{eq: dist via BX}
\|\bar{B}X\|_2 \le \|BX\|_2 = d(DX,E_k).
\end{equation}

It remains to bound $\|\bar{B}X\|_2$ below. This can be done using Theorem~\ref{thm: concentration}(ii):
\begin{equation}				\label{eq: barB X}
\Pr{ \|\bar{B}X\|_2 < c \|\bar{B}\|_\HS } \le 2 \exp \Big( - \frac{c \|\bar{B}\|_\HS^2}{\|\bar{B}\|^2} \Big).
\end{equation}
For $i > k_1-k$, Cauchy interlacing theorem yields
$s_i(\bar{B}) = s_i(B) \ge s_{i+k}(D)$, thus
\begin{align*}
\|\bar{B}\|_\HS^2
= \sum_{i=1}^n s_i(\bar{B})^2
&\ge (k_1-k) s_{k_1-k}(B)^2 + \sum_{i>k_1-k} s_{i+k}(D)^2 \\
&= (k_1-k) s_{k_1-k}(B)^2 + \bar{S}_{k_1}^2.
\end{align*}
Further, $\|\bar{B}\| = \max_i s_i(\bar{B}) = s_{k_1-k}(B)$.
In particular, $\|\bar{B}\|_\HS^2 \ge \bar{S}_{k_1}^2$ and $\|\bar{B}\|_\HS^2 / \|\bar{B}\|^2 \ge k_1-k$.
Putting this along with \eqref{eq: dist via BX} into \eqref{eq: barB X}, we complete the proof of part (i).

(ii)
We truncate the singular value decomposition $D = \sum_{i=1}^n s_i u_i v_i^*$ by defining
$$
D_0 = \sum_{i=1}^{k_0} s_i u_i v_i^*, \quad
\bar{D} = \sum_{i=k_0+1}^n s_i u_i v_i^*.
$$
By the triangle inequality, we have
\begin{equation}				\label{eq: dist 2 terms}
d(DX, E_k) \le d(D_0 X, E_k) + \|\bar{D} X\|_2.
\end{equation}
We will estimate these two terms separately.

The second term, $\|\bar{D} X\|_2$, can be bounded using sub-gaussian concentration,
Theorem~\ref{thm: concentration}(i).
Since $\|\bar{D}\|=s_{k_0+1}$ and $\|\bar{D}\|_\HS = \bar{S}_{k_0}$, it follows that
$$
\Pr{ \|\bar{D} X\|_2 > C \bar{S}_{k_0} + t } \le 2 \exp \big(-ct^2 / s_{k_0+1}^2),
\quad t \ge 0.
$$
Using this for $t = \sqrt{k} s_{k_0+1}$, we obtain that with probability
at least $1 - 2 \exp(-ck)$,
\begin{equation}				\label{eq: Dbar X small}
\|\bar{D} X\|_2 \le C(\bar{S}_{k_0} + \sqrt{k} \, s_{k_0+1}).
\end{equation}

Next, we estimate the first term in \eqref{eq: dist 2 terms}, $d(D_0 X, E_k)$.
Our immediate goal is to represent $D_0 X$ as a linear combination
\begin{equation}				\label{eq: D0X sum}
D_0 X = \sum_{i=1}^k a_i D_0 X_i
\end{equation}
with some control of the norm of the coefficient vector
$a=(a_1,\ldots,a_k)$.
To this end, let us consider the singular value decomposition
$$
D_0 = U_0 \Sigma_0 V_0^*; \quad \text{denote } P_0 = V_0^*.
$$
Thus $P_0$ is a $k_0 \times n$ matrix satisfying $P_0 P_0^* = I_{k_0}$.
Let $G$ denote the $n \times k$ with columns $X_1, \ldots, X_k$.

We apply Theorem~\ref{thm: rectangular} for the $k_0 \times k$ matrix $P_0 G$.
It states that with probability at least $1 - 2k \exp(-c(k-k_0))$, we have
\begin{equation}				\label{eq: sk0}
s_{k_0} (P_0 G) \ge c \, \frac{k-k_0}{k} =: \s_0.
\end{equation}
Using Lemma~\ref{lem: smallest sv}(ii) we can find a coefficient vector $a=(a_1,\ldots,a_k)$ such that
\begin{gather}
P_0 X = P_0 G a = \sum_{i=1}^k a_i P_0 X_i,  			\label{eq: P0X sum} \\
\quad \|a\|_2 \le s_{k_0} (P_0 G)^{-1} \|P_0 X\|_2 \le \s_0^{-1} \|P_0 X\|_2.   \label{eq: a small prelim}
\end{gather}
Multiplying both sides of \eqref{eq: P0X sum} by $U_0 \Sigma_0$
and recalling that $D_0 = U_0 \Sigma_0 V_0^* = U_0 \Sigma_0 P_0$,
we obtain the desired identity \eqref{eq: D0X sum}.

To finalize estimating $\|a\|_2$ in \eqref{eq: a small prelim}, recall that
$\|P_0\|_\HS^2 = \tr(P_0 P_0^*) = \tr(I_{k_0}) = k_0$
and $\|P_0\| = 1$. Then Theorem~\ref{thm: concentration}(i) yields
that with probability at least $1-2\exp(-c k_0)$, one has $\|P_0 X\|_2 \le C \sqrt{k_0}$.
Intersecting with the event \eqref{eq: a small prelim}, we conclude that with probability at least
$1 - 4k \exp(-c(k-k_0))$, one has
\begin{equation}				\label{eq: a small}
\|a\|_2 \le C \s_0^{-1} \sqrt{k_0} =: M.
\end{equation}

Now we have representation \eqref{eq: D0X sum} with a good control of $\|a\|_2$.
Then we can estimate the distance as follows:
\begin{align*}
d(D_0X,E_k)
&= \inf_{z \in E_k} \|D_0X - z\|_2
\le \Big\| \sum_{i=1}^k a_i D_0X_i - \sum_{i=1}^k a_i DX_i \Big\|_2 \\
&= \Big\| \sum_{i=1}^k a_i \bar{D} X_i \Big\|_2
\le \|a\|_2 \|\bar{D} G\|.
\end{align*}
(Recall that $G$ denotes the $n \times k$ with matrix columns $X_1, \ldots, X_k$.)
Applying Theorem~\ref{thm: product norm}, we have with probability at least
$1-2\exp(-k)$ that
$$
\|\bar{D} G\| \le C(\|\bar{D}\|_\HS + \sqrt{k} \|\bar{D}\|) = C(\bar{S}_{k_0} + \sqrt{k} \, s_{k_0+1}).
$$
Intersecting this with the event \eqref{eq: a small}, we obtain with probability
at least $1 - 6k \exp(-c(k-k_0))$ that
$$
d(D_0X,E_k) \le C M (\bar{S}_{k_0} + \sqrt{k} \, s_{k_0+1}).
$$

Finally, we combine this with the event \eqref{eq: Dbar X small} and put into
the estimate \eqref{eq: dist 2 terms}. It follows that with probability
at least $1 - 8k \exp(-c(k-k_0))$, one has
$$
d(DX, E_k) \le C(M+1) (\bar{S}_{k_0} + \sqrt{k} s_{k_0+1}).
$$
Due to our choice of $M$ (in \eqref{eq: a small} and \eqref{eq: sk0}),
the theorem is proved.\footnote{The factor $8$ in the probability estimate
  can be reduced to $2$ by adjusting $c$.
  We will use the same step in later arguments.}
\end{proof}

\section{Construction of a test projection}				\label{s: test projection}
%--------------

We are now ready to construct a test projection $P$,
which will be used later in Proposition~\ref{prop: deloc from P}.

\begin{theorem}[Test projection]					\label{thm: test projection}
  Let $1 \le l \le n/4$ and $z \in \C$, $|z| \le K_1 \sqrt{n}$ for some absolute constant $K_1$.
  Consider a random matrix $G$ as in Theorem~\ref{thm: main approximate},
  and let $A = G - z I_n$. Let $A_j$ denote the columns of $A$.
  Then one can find an integer $l' \in [l/2, l]$ and an $l' \times n$ matrix $P$
  in such a way that $l'$ and $P$ are determined by $l$, $n$ and $\{A_j\}_{j>l}$, and
  so that the following properties hold:
  \begin{enumerate}[(i)]
    \item $P P^* = I_{l'}$;
    \item $\ker P \supseteq \{A_j\}_{j>l}$;
    \item with probability at least $1 - 2 n^2 \exp(-cl/\log n)$, one has
      \begin{align*}
      \|P e_i\|_2 &\le C \sqrt{l} \, \log^{3/2} n \cdot \|P e_j\|_2 \quad \text{for } 1 \le i,j \le l' ;  \\
      \|P e_i\|_2 &= 0  \quad \text{for } l' < i \le l.
      \end{align*}
  \end{enumerate}
  Here $C = C(K,K_1)$, $c = c(K,K_1)>0$.
\end{theorem}

In the rest of this section we prove Theorem~\ref{thm: test projection}.

\subsection{Selection of the spectral window $l'$}
%.......................................................

Consider the $n \times n$ random matrix $A$ with columns $A_j$.
Let $\bar{A}$ denote the $(n-l) \times (n-l)$ minor of $A$ obtained by removing
the first $l$ rows and columns.
By known invertibility results for random matrices, we will see that
most singular values of $\bar{A}$, and thus also of $\bar{A}^{-1}$,
are within a factor $n^{O(1)}$ from each other.
Then we will find a somewhat smaller interval (a ``spectral window'')
in which the singular values of $\bar{A}^{-1}$ are within constant factor from each other.
This is a consequence of the following elementary lemma.

\begin{lemma}[Improving the regularity of decay]									\label{lem: ratios}
  Let $s_1 \ge s_2 \ge \cdots \ge s_n$, and define
  $\bar{S}_k^2 = \sum_{j > k} s_j^2$ for $k \ge 0$.
  Assume that for some $l \le n$ and $R \ge 1$, one has
  \begin{equation}				\label{eq: ratio assumption}
  \frac{s_{l/2}}{s_l} \le R.
  \end{equation}
  Set $\d = c / \log R$.
  Then there exists $l' \in [l/2, l]$ such that
  \begin{equation}				\label{eq: ratio conclusion}
  \frac{s_{(1-\d)\,l'}^2}{s_{(1+\d)\,l'}^2} \le 2
  \quad \text{and } \quad
  \frac{\bar{S}_{(1-\d)\,l'}^2}{\bar{S}_{(1+\d)\,l'}^2} \le 5.
  \end{equation}
\end{lemma}

\begin{proof}
Let us divide the interval $[l/2,l]$ into $1/(8\d)$ intervals of length $4\d l$.
Then for at least one of these intervals, the sequence $s_i^2$
decreases over it by a factor at most $2$. Indeed, if this were not true, the sequence would decrease
by a factor at least $2^{1/(8\d)} > R$ over $[l/2,l]$,
which would contradict the assumption \eqref{eq: ratio assumption}.
Set $l'$ to be the midpoint of the interval we just found, thus
\begin{equation}				\label{eq: ratio larger interval}
\frac{s_{l'-2\d l}^2}{s_{l'+2\d l}^2} \le 2.
\end{equation}
By monotonicity of $s_i^2$, this implies the first part of the conclusion \eqref{eq: ratio conclusion}.
To see this, note that since $l' \le l$, we have $l'-2\d l \le (1-\d)l' \le (1+\d)l' \le l'+2\d l$.

To deduce the second part of \eqref{eq: ratio conclusion}, note that by
monotonicity we have
\begin{align}
\bar{S}_{l'-\d l}^2
  &= \sum_{l'-\d l < i \le l'+\d l} s_i^2 + \bar{S}_{l'+\d l}^2
  \le 2\d l \cdot s_{l'-2\d l}^2 + \bar{S}_{l'+\d l}^2, 		\label{eq: S upper}\\
\bar{S}_{l'+\d l}^2
  &\ge \sum_{l'+\d l < i \le l'+2\d l} s_i^2
  \ge \d l \cdot s_{l'+2\d l}^2
  \ge \frac{1}{2} \d l \cdot s_{l'-2\d l}^2,					\label{eq: S lower}
\end{align}
where the very last inequality follows from \eqref{eq: ratio larger interval}.
Estimates \eqref{eq: S upper} and \eqref{eq: S lower} together imply that
$\bar{S}_{l'-\d l}^2 \le 5 \bar{S}_{l'+\d l}^2$.
Like in the first part, we finish by monotonicity.
\end{proof}

We shall apply Lemma~\ref{lem: ratios} to the singular values of $\bar{A}^{-1}$, i.e. for
$$
s_j = s_j(\bar{A}^{-1}), \quad \bar{S}_k^2 = \sum_{j > k} s_j^2.
$$
To verify the assumptions of the lemma, we can use known estimates
of the extreme singular values of random matrices.
By Theorem~\ref{thm: product norm} (see Remark~\ref{rem: product norm special cases}),
with probability at least $1-\exp(-n)$, we have $\|G\| \le C \sqrt{n}$, and thus
$$
s_l^{-1} \le s_n^{-1}
= s_1(\bar{A})
= \|\bar{A}\|
\le \|A\|
= \|G - z I_n\|
\le \|G\| + |z|
\le (C+K_1) \sqrt{n}.
$$
Further, by Theorem~\ref{thm: intermediate}, with probability
at least $1 - 2 l \exp(-c(n-2l))$, one has
$$
s_{l/2}^{-1} = s_{n-l/2+1}(\bar{A}) \ge \frac{c}{\sqrt{n}}.
$$
(Here we used that $l \le n/4$.)
Summarizing, with probability at least $1-2n\exp(-cl)$,
\begin{equation}				\label{eq: Abar conditioned}
\frac{c_1}{\sqrt{n}} \le s_l \le s_{l/2} \le C_1 \sqrt{n}.
\end{equation}

Let us condition on $\bar{A}$ for which event \eqref{eq: Abar conditioned} holds.
We apply Lemma~\ref{lem: ratios} with $R  = (C_1/c_1) n$ and thus for
\begin{equation}				\label{eq: delta}
\d = c/\log n.
\end{equation}
We find $l' \in [l/2, l]$ such that \eqref{eq: ratio conclusion} holds.
Note that the value of $l'$ depends only on the minor $\bar{A}$, thus
only on $\{A_j\}_{j > l}$, as claimed in Theorem~\ref{thm: test projection}.
Since we have conditioned on $\bar{A}$, the value of the ``spectral window''
$l'$ is now fixed.

\subsection{Construction of $P$}						
%.........................

We construct $P$ in two steps. First we define a matrix $Q$ of the same dimensions
that satisfies (ii) of the Theorem, and then obtain $P$ by orthogonalization of the rows of $Q$.

Thus we shall look for an $l' \times n$ matrix $Q$ that consists of three blocks of columns:
$$
Q =
\left[ \begin{array}{cccc|ccc|c}
  q_{11} & 0 & \cdots & 0 & 0 & \cdots & 0 & \bar{q}_1^\tran \\
  0 & q_{22} & \cdots & 0 & 0 & \cdots & 0 & \bar{q}_2^\tran \\
  \vdots & \vdots & \ddots & \vdots & \vdots & \vdots & \ddots & \vdots \\
  0 & 0 & \cdots & q_{l' l'} & 0 & \cdots & 0 & \bar{q}_{l'}^\tran
\end{array} \right],
\quad q_{ii} \in \C, \; \bar{q}_i \in \C^{n-l}.
$$
We require that $Q$ satisfy condition (ii) in Theorem~\ref{thm: test projection}, i.e. that
\begin{equation}							\label{eq: ker Q}
\ker Q \supseteq \{A_j\}_{j>l}.
\end{equation}
We explore this requirement in Section~\ref{s: ker Q}; for now let us
assume that it holds.

Choose $P$ to be an $l' \times n$ matrix that satisfies the following two defining
properties:
\begin{enumerate}[(a)]
  \item $P$ has orthonormal rows;
  \item the span of the rows of $P$ is the same as the span of the rows of $Q$.
\end{enumerate}
One can construct $P$ by Gram-Schmidt orhtogonalization of the rows of $Q$.

Note that the construction of $P$ along with \eqref{eq: ker Q}
implies (i) and (ii) of Theorem~\ref{thm: test projection}.
It remains to estimate $\|Pe_j\|_2$ thereby proving (iii) of Theorem~\ref{thm: test projection}.

\subsection{Reducing $\|Pe_i\|_2$ to distances between random vectors and subspaces}
%.........................

\begin{proposition}[Norms of columns of $P$ via distances]				\label{prop: Pei via distances}
  Let $q_i$ denote the rows of $Q$ and $q_{ij}$ denote the entries of $Q$. Then:
  \begin{enumerate}[(i)]
    \item The values of $\|Pe_i\|_2$, $i \le n$, are determined by $Q$, and they do not depend
      on a particular choice of $P$ satisfying its defining properties (a), (b).
    \item For every $i \le l'$,
      \begin{equation}							\label{eq: Pei}
      \|Pe_i\|_2 = \frac{|q_{ii}|}{d(q_i, E_i)}, \quad \text{where } E_i = \Span(q_j)_{j \le l', \, j \ne i}.
      \end{equation}
    \item For every $l' < i \le l$, $\|Pe_i\|_2 = 0$.
  \end{enumerate}
\end{proposition}

\begin{proof}
(i) Any $P, P'$ that satisfy the defining properties (a), (b) must satisfy
$P' = UP$ for some $l' \times l'$ unitary matrix $U$. It follows that $\|P'e_i\|_2 = \|Pe_i\|_2$ for all $i$.

\medskip

(ii)
Let us assume that $i=1$; the argument for general $i$ is similar.
By part (i), we can construct the rows of $P$ by performing Gram-Schmidt procedure on the
rows of $Q$ in any order. We choose the following order: $q_{l'}, q_{l'-1},\ldots,q_1$,
and thus construct the rows $p_{l'}, p_{l'-1},\ldots,p_1$ of $P$.
This yields
\begin{align}
p_1 &= \frac{\tilde{p}_1}{\|\tilde{p}_1\|_2}, \quad \text{where } \tilde{p}_1 = q_1 - P_{E_1} q_1   \label{eq: p1} \\
p_j &\in \Span(q_k)_{k \ge j}, \quad j = 1,\ldots, l'			\label{eq: pj}
\end{align}
Recall that we would like to estimate
\begin{equation}							\label{eq: Pe1}
\|P e_1\|_2^2 = |p_{11}|^2 + |p_{21}|^2 \cdots + |p_{l' 1}|^2
\end{equation}
where $p_{ij}$ denote the entries of $P$.

First observe that all vectors in $E_1 = \Span(q_k)_{k \ge 2}$ have their
first coordinate equal zero, because the same holds for the vectors
$q_k$, $k \ge 2$, by the construction of $Q$.
Since $P_{E_1} q_1 \in E_1$, this implies by \eqref{eq: p1} that
$\tilde{p}_{11} = q_{11}$.
Further, again by \eqref{eq: p1} we have $\|\tilde{p}_1\|_2 = d(q_1,E_1)$. Thus
$$
p_{11} = \frac{\tilde{p}_{11}}{\|\tilde{p}_1\|_2} = \frac{q_{11}}{d(q_1,E_1)}.
$$

Next, for each $2 \le j \le l'$, \eqref{eq: pj} implies that
$p_j \in \Span(q_k)_{k \ge 2} = E_1$,
and thus the first coordinate of $p_j$ equal zero.
Using this in \eqref{eq: Pe1}, we conclude that
$$
\|Pe_1\|_2 = |p_{11}| =  \frac{|q_{11}|}{d(q_1,E_1)}.
$$
This completes the proof of (ii).

\medskip

(iii) is trivial since $Q e_i = 0$ for all $l' < i \le l$ by the construction of $Q$,
while the rows of $P$ are the linear combination of the rows of $Q$.
\end{proof}

\subsection{The kernel requirement \eqref{eq: ker Q}}			\label{s: ker Q}
%.........................

In order to estimate the distances $d(q_i, E_i)$ defined by the rows of $Q$,
let us explore the condition \eqref{eq: ker Q} for $Q$.
To express this condition algebraically, let us consider the $n \times (n-l)$ matrix $A^{(l)}$
obtained by removing the first $l$ columns from $A$.
Then \eqref{eq: ker Q} can be written as
\begin{equation}							\label{eq: QAl}
Q A^{(l)} = 0.
\end{equation}
Let us denote the first $l$ rows of $A^{(l)}$ by $B_i^\tran$, thus
\begin{equation}							\label{eq: A without l columns}
A^{(l)} =
\newcommand*{\hsp}{\hspace{1cm}}
\left[ \begin{array}{c}
  \hsp B_1^\tran \hsp \\
  \hsp \vdots \hsp \\
  \hsp B_l^\tran \hsp \\
  \cline{1-1}
  \medskip \\
  \hsp \bar{A} \hsp
  \bigskip
\end{array} \right],
\quad B_i \in \C^{n-l}.
\end{equation}
Then \eqref{eq: QAl} can be written as
$$
q_{ii} B_i^\tran + \bar{q}_i^\tran \bar{A} = 0, \quad i \le l'.
$$
Without loss of generality, we can assume that the matrix $\bar{A}$ is almost surely invertible
(see Section \ref{s: preliminaries} for a perturbation argument achieving this).
Multiplying both sides of the previous equations by $\bar{A}^{-1}$,
we further rewrite them as
\begin{equation}							\label{eq: bar qi}
\bar{q}_i = -q_{ii} D B_i, \quad i \le l',
\qquad \text{where } D := (\bar{A}^{-1})^\tran.
\end{equation}
Thus we can choose $Q$ to satisfy the requirement \eqref{eq: ker Q} by
choosing $q_{ii} > 0$ arbitrarily and
defining $\bar{q}_i$ as in \eqref{eq: bar qi}.

\subsection{Estimating the distances, and completion of proof of Theorem~\ref{thm: test projection}}  \label{completion}
%.........................

We shall now estimate $\|Pe_i\|_2$, $1 \le i \le l'$, using identities \eqref{eq: Pei} and \eqref{eq: bar qi}.
By the construction of $Q$ and \eqref{eq: bar qi} we have
$$
q_i = (0 \cdots q_{ii} \cdots 0 \; \bar{q}_i^\tran)
= - q_{ii} r_i,
\quad \text{where} \quad
r_i = (0 \cdots -1 \cdots 0 \; (DB_i)^\tran).
$$
Let us estimate $\|Pe_1\|_2$; the argument for general $\|Pe_i\|_2$ is similar.
By \eqref{eq: Pei},
\begin{equation}							\label{eq: d1}
\|P e_1\|_2 = \frac{|q_{11}|}{d(q_{11} r_1, \Span(q_{jj} r_j)_{2 \le j \le l'})}
= \frac{1}{d(r_1, \Span(r_j)_{2 \le j \le l'})} =: \frac{1}{d_1}.
\end{equation}
We will use Theorem~\ref{thm: distance} to obtain lower and upper bounds on $d_1$.

\subsubsection{Lower bound on $d_1$.}
%,,,,,,,,,,,

By the definition of $r_j$, we have
$$
d_1 \ge \sqrt{1 + d(DB_1, \Span(DB_j)_{2 \le j \le l'})^2}.
$$
We apply Theorem~\ref{thm: distance} in dimension $n-l$ instead of $n$, and with
$$
k=l'-1,\quad k_0=(1-\d) l', \quad k_1=(1+\d)l'.
$$
Recall here that in \eqref{eq: delta} we selected $\d = c/\log n$.
Note that by construction \eqref{eq: A without l columns},
the vectors $B_i$ do not contain the diagonal elements of $A$, and so their entries have
mean zero as required in Theorem~\ref{thm: distance}. Applying part (i) of that theorem, we obtain
with probability at least $1 - 2 \exp(- c \d l')$ that
\begin{equation}							\label{eq: d1 lower}
d_1 \ge \sqrt{ 1 + c \bar{S}_{k_1}^2}
\ge \frac{1}{2}(1+ c \bar{S}_{k_1}).
\end{equation}

\subsubsection{Upper bound on $d_1$.}
%,,,,,,,,,,,

Now we apply part (ii) of Theorem~\ref{thm: distance}. This time we shall use
a sharper bound stated in Remark~\ref{rem: distance sharper}.
It yields that with probability at least $1-2 l' \exp(-c \d l')$, the following holds.
There exists $a = (a_2,\ldots, a_{l'})$ such that
\begin{gather}				\label{eq: dist sum prelim}
\Big\| DB_1 - \sum_{j=2}^{l'} a_j DB_j \Big\|_2
\le C M (\bar{S}_{k_0} + \sqrt{k} \, s_{k_0+1}), \\			
\|a\|_2 \le M,
\quad \text{where }
M = \frac{C k \sqrt{k_0}}{k-k_0} \le 2\sqrt{l'} / \d.    \label{eq: M dist}
\end{gather}

We can simplify \eqref{eq: dist sum prelim}. Using \eqref{eq: ratio conclusion} and monotonicity,
we have
$$
k s_{k_0+1}^2
\le 2 k s_{k_1}^2
= \frac{2k}{k_1-k_0} (k_1-k_0) s_{k_1}^2
\le \frac{2k}{k_1-k_0} \bar{S}_{k_0}^2
\le \frac{2}{\d} \bar{S}_{k_0}^2,
$$
thus again using \eqref{eq: ratio conclusion}, we have
$$
(\bar{S}_{k_0} + \sqrt{k} \, s_{k_0+1})^2
\le 2( \bar{S}_{k_0}^2 + k s_{k_0+1}^2)
\le \frac{6}{\d} \bar{S}_{k_0}^2
\le \frac{30}{\d} \bar{S}_{k_1}^2.
$$
Hence \eqref{eq: dist sum prelim} yields
$$
\Big\| DB_1 - \sum_{j=2}^{l'} a_j DB_j \Big\|_2
\le \frac{C M}{\sqrt{\d}} \bar{S}_{k_1}.
$$
Recall that this holds with probability at least $1-2 l' \exp(-c \d l')$.
On this event, by the construction of $r_i$ and using the bound on $a$ in \eqref{eq: M dist}, we have
\begin{align}
d_1
  &=d(r_1, \Span(r_j)_{2 \le j \le l'})
  \le \Big\| r_1 - \sum_{j=2}^{l'} a_j r_j \Big\|_2	\nonumber\\
  &= 1 + \|a\|_2 + \Big\| DB_1 - \sum_{j=2}^{l'} a_j DB_j \Big\|_2
  \le 2 M \Big( 1 + \frac{C}{\sqrt{\d}} \bar{S}_{k_1}	\Big).     			\label{eq: d1 upper}
\end{align}

\subsubsection{Completion of the proof of Theorem~\ref{thm: test projection}}
%,,,,,,,,,,,

Combining the events \eqref{eq: d1 upper} and \eqref{eq: d1 lower}, we have shown the following.
With probability at least $1 - 4 l' \exp(- c \d l')$, the following two-sided estimate holds:
$$
\frac{1}{2}(1 + c \bar{S}_{k_1}) \le d_1 \le 2 M \Big( 1 + \frac{C}{\sqrt{\d}} \bar{S}_{k_1}	\Big).
$$
A similar statement can be proved for general $d_i$, $1 \le i \le l'$. By intersecting
these events, we obtain that with probability at least
$1 - 4 (l')^2 \exp(- c \d l')$, all such bounds for $d_i$ hold simultaneously.
Suppose this indeed occurs. Then by \eqref{eq: d1}, we have
\begin{equation}				\label{eq: ratio Pei Pej}
\frac{\|Pe_i\|_2}{\|Pe_j\|_2}
= \frac{d_j}{d_i}
\le \frac{4 M \big( 1 + (C/\sqrt{\d}) \bar{S}_{k_1} \big)}{1 + c \bar{S}_{k_1}}
\le \frac{C_1}{\sqrt{\d}} \, M
\quad \forall 1 \le i, j \le l'.
\end{equation}

We have calculated the conditional probability of \eqref{eq: ratio Pei Pej}; recall that
we conditioned on $\bar{A}$ which satisfies the event \eqref{eq: Abar conditioned}, which itself holds
with probability $1-2n\exp(-cl)$.
Thus the unconditional probability of the event \eqref{eq: ratio Pei Pej} is at least
$1-2n\exp(-cl) - C_1 (l')^2 \exp(- c \d l')$.
Recalling that $l/2 \le l \le n/4$ and $\d = c/\log n$, and simplifying this expression,
we arrive at the probability bound claimed in Theorem~\ref{thm: test projection}.
Since $M \le 2 \sqrt{l}/\d$ according to \eqref{eq: M dist},
the estimate \eqref{eq: ratio Pei Pej} yields the first part of (iii) in Theorem~\ref{thm: test projection}.
The second part, stating that $Pe_i = 0$ for $l' < i \le l$,
was already noted in (iii) or Proposition~\ref{prop: Pei via distances}.
Thus Theorem~\ref{thm: test projection} is proved.
\qed

\section{Proof of Theorem~\ref{thm: main approximate} and Corollary~\ref{cor: psi_alpha}}				\label{s: proof}
%--------------

Let $G$ be a random matrix from Theorem~\ref{thm: main approximate}. We shall apply
Proposition~\ref{prop: deloc from P} for
\begin{equation}    \label{eq: bound on |z|}
A = G - z I_n, \quad |z| \le K_1 \sqrt{n},
\end{equation}
where $z \in \C$ is a fixed number for now, and $K_1$ is a parameter to be chosen later.
The power of Proposition~\ref{prop: deloc from P}
relies on the existence of a test projection $P$ for which the balancing event
$\BB_{\a, \kappa}$ is likely.
We are going to validate this condition using the test projection
constructed in Theorem~\ref{thm: test projection}.

\begin{proposition}[Balancing event is likely]			\label{prop: balancing likely}
  Let $\a = c / (l \log^{3/2} n)$ and $\kappa=c$.
  Then, for every fixed $(j_0, J_0) \in \Lambda$, one can find a test projection
  as required in Proposition~\ref{prop: deloc from P}. Moreover,
  $$
  \Pr[A]{ \BB_{\a,\kappa} } \ge 1 - 2 n^2 \exp(-cl/\log n).
  $$
  Here $c = c(K,K_1) > 0$.
\end{proposition}

\begin{proof}
Without loss of generality, we assume that $j_0 = 1$ and $J_0 = \{2,\ldots,n\}$.
We apply Theorem~\ref{thm: test projection},
and choose $l' \in [l/2,l]$ and $P$ determined by $\{A_j\}_{j > l}$ guaranteed by that theorem.
The test projeciton $P$ automatically satisfies the conditions of Proposition~\ref{prop: deloc from P}.
Moreover, with probability at least $1 - 2 n^2 \exp(-cl/\log n)$, one has
\begin{equation}				\label{eq: Pej Pe1}
\|Pe_j\|_2 \le C \sqrt{l} \, \log^{3/2} n \cdot \|Pe_1\|_2 \quad \text{for } 2 \le j \le l.
\end{equation}
Let us condition on $\{A_j\}_{j > l}$ for which the event \eqref{eq: Pej Pe1} holds;
this fixes $l'$ and $P$ but leaves $\{A_j\}_{j \le l}$ random as before.

The definition \eqref{eq: balancing} of balancing event $\BB_{\a,\kappa}$ requires us to estimate the norms of
$$
PA_1 = P G_1 - z P e_1 \quad \text{and} \quad P A_{J_0} = P G_{J_0} - z P_{J_0}.
$$
For $PA_1$, we use the small ball probability estimate, Theorem~\ref{thm: concentration}(ii).
Recall that $\|P\|_\HS^2 = \tr(PP^*) = \tr(I_{l'}) = l' \ge l/2$ and $\|P\| = 1$.
It follows that with probability at least $1-2\exp(-cl)$, we have
\begin{equation}				\label{eq: PA1}
\|PA_1\|_2 \ge c(\sqrt{l} + |z| \|Pe_1\|_2).
\end{equation}

Next, we estimate
\begin{equation}							\label{eq: PAJ0 two terms}
\|P A_{J_0}\| \le \|P G_{J_0}\| + |z| \|P_{J_0}\|.
\end{equation}
For the $l' \times (l-1)$ matrix $PG_{J_0}$, Theorem~\ref{thm: product norm}
(see Remark~\ref{rem: product norm special cases}) implies that
with probability at least $1-2\exp(-l)$ one has
$
\|PG_{J_0}\| \le C \sqrt{l}.
$
Further, \eqref{eq: Pej Pe1} allows us to bound
$
\|P_{J_0}\|
\le \|P_{J_0}\|_\HS
\le \sqrt{l} \max_{2 \le j \le l} \|Pe_j\|_2
\le C l \, \log^{3/2} n \cdot \|Pe_1\|_2.
$
Thus \eqref{eq: PAJ0 two terms} yields
\begin{equation}				\label{eq: PAJ0}
\|P A_{J_0}\| \le C \sqrt{l} + C l \log^{3/2} n \cdot |z| \|Pe_1\|_2.
\end{equation}

Hence, estimates \eqref{eq: PA1} and \eqref{eq: PAJ0} hold simultaneously
with probability at least $1 - 4\exp(-cl)$. Recall that this concerns conditional probability,
where we conditioned on the event \eqref{eq: Pej Pe1}, which itself holds
with probability at least $1 - 2 n^2 \exp(-cl/\log n)$.
Therefore, estimates \eqref{eq: PA1} and \eqref{eq: PAJ0} hold simultaneously
with (unconditional) probability at least
$1 - 4\exp(-cl) - 2 n^2 \exp(-cl/\log n) \ge 1 - 6 n^2 \exp(-cl/\log n)$.
Together they yield
$$
\|PA_1\|_2 \ge \a \|PA_{J_0}\| \quad \text{where } \a = c / (l \log^{3/2} n).
$$
This is the first part of the event $\BB_{\a,\kappa}$.
Finally, \eqref{eq: PA1} implies that $\|PA_1\|_2 \ge c \sqrt{l}$,
which is the second part of the event $\BB_{\a,\kappa}$ for $\kappa=c$.
The proof is complete.
\end{proof}

Substituting the conclusion of Proposition~\ref{prop: balancing likely}
into Proposition~\ref{prop: deloc from P}, we obtain:

\begin{proposition}				\label{prop: localization unlikely}
  Let $0<w<l$ and $W = C l \log^{3/2} n$. Then
  $$
  \Pr{\LL_{W,w}} \le
   4 n^3 \exp(-cl/\log n).  \qquad 
  $$
  Here $C=C(K,K_1)$, $c=c(K,K_1)>0$.
\end{proposition}

From this we can readily deduce a slightly stronger version of Theorem~\ref{thm: main approximate}.

\begin{corollary}					\label{cor: delocalization all z}
  Consider a random matrix $G$ as in Theorem~\ref{thm: main approximate}.
  Let $0 \le s \le n$, $s+2 \le l \le n/4$ and $W = C l \log^{3/2} n$.
  Then the event
  $$
  \LL_W := \Big\{ \exists \text{ $\Big(\frac{s}{\sqrt{n}}\Big)$-approximate 
    eigenvector $v$ of $G$ with }
    \|v\|_2=1, \; \|v\|_\infty > W \sqrt{\frac{l}{n}} \Big\}
  $$
  is unlikely:
  $$
  \Pr{\LL_W} \le C n^5 \exp(-cl/\log n).
  $$
  Here $C = C(K)$, $c = c(K)>0$.
\end{corollary}

\begin{proof}
Recall that $G$ is nicely bounded with high probability. Indeed, Theorem~\ref{thm: product norm}
(see Remark~\ref{rem: product norm special cases}) states that the event
\begin{equation}				\label{eq: EEnorm}
\EE_\textrm{norm} := \left\{ \|G\| \le C_1 \sqrt{n} \right\}
\quad \text{is likely:} \quad
\Pr{\EE_\textrm{norm}} \le 1 - 2 \exp(-cn).
\end{equation}
Assume that $\EE_\textrm{norm}$ holds. Then all $(s/\sqrt{n})$-approximate eigenvalues 
of $G$ are contained in the complex disc centered at the origin and with radius 
$\|G\|+ s/\sqrt{n} \le 2C_1 \sqrt{n}$.
Let $\{z_1,\ldots,z_N\}$ be a $(1/\sqrt{n})$-net of this disc such that $N \le C_2 n^2$.

Assume $\LL_W$ holds, so there exists an $(s/\sqrt{n})$-approximate 
eigenvector $v$ of $G$ such that
$\|v\|_2=1$ and $\|v\|_\infty > W \sqrt{l/n}$. Choose a point $z_i$ in the net closest to $z$,
so $|z-z_i| \le 1/\sqrt{n}$.
Then
$$
\|(G-z_i I_n)v\|_2 
\le \|(G-zI_n)v\|_2 + |z-z_i|
\le \frac{s+1}{\sqrt{n}}.
$$
This argument shows that $\LL_W \cap \EE_\textrm{norm} \subseteq \bigcup_{i=1}^N \LL_W^{(i)}$,
where
$$
\LL_W^{(i)} =
\Big\{ \exists v \in S^{n-1}:
  \|v\|_\infty > W \sqrt{\frac{l}{n}}
  \text{ and }
  \|(G-z_i I_n)v\|_2 \le \frac{s+1}{\sqrt{n}} \Big\}.
$$
Recall that the probability of $\EE_\textrm{norm}$ is estimated in \eqref{eq: EEnorm},
and the probabilities of the events $\LL_W^{(i)}$ can be bounded using
Proposition~\ref{prop: localization unlikely} with $w=s + 1$. 
(Our assumption that $l \ge s+2$ enforces the bound $w < l$ that 
is needed in Proposition~\ref{prop: localization unlikely}.)
It follows that
$$
\Pr{\LL_W}
\le \Pr{\EE_\textrm{norm}^c} + \sum_{i=1}^N \Pr{\LL_W^{(i)}}
\le 2 \exp(-cn) + C_2 n^2 \cdot 4 n^3 \exp(-cl/\log n).
$$
Simplifying this bound we complete the proof.
\end{proof}

\medskip

\begin{proof}[Proof of Theorem~\ref{thm: main approximate}]
We are going to apply Corollary~\ref{cor: delocalization all z} for $l = C t (s+1) \log^2 n$.
This is possible as long as $s \le n$ and $t (s+1) < c n / \log^2 n$,
since the latter restriction enforces the bound $l \le n/4$.
In this regime, the conclusion of Theorem~\ref{thm: main approximate} 
follows directly from Corollary~\ref{cor: delocalization all z}. 

In the remaining case, where either $s \ge n$ or $t (\e\sqrt{n}+1) > c n / \log^2 n$,
the right hand side of \eqref{eq: main approximate} 
is greater than $\|v\|_2$ for an appropriate choice of the constant $C$.
Thus, in this case, the bound \eqref{eq: main approximate} holds trivially 
since one always has $\|v\|_\infty \le \|v\|_2$. 
Theorem~\ref{thm: main approximate} is proved.
\end{proof}

\subsection{Deduction of Corollary~\ref{cor: psi_alpha}}
%...............

Using a standard truncation argument, we will now
deduce Corollary~\ref{cor: psi_alpha} for general exponential tail decay.
We will first prove the following relaxation of Proposition~\ref{prop: localization unlikely}.

\begin{proposition}				\label{prop: localization unlikely-general}
  Let $G$ be an $n \times n$ real random matrix whose entries $G_{ij}$  are independent
  random variables satisfying $\E G_{ij} = 0$, $\E G_{ij}^2 \ge 1$ and $\|G_{ij}\|_{\psi_\a} \le M$.
  Let $z \in \mathbb{C}$, $0<w<l-1$, and  $t \ge 2$. Set $W = C l t^{\b} \log^{\g} n$, and consider the event $\LL_{W,w}$ defined as in \eqref{eq: LL} for the matrix $A=G-zI_n$.
  Then
  $$
  \Pr{\LL_{W,w}} \le
   4 n^3 \exp(-cl/\log n) +n^{-t}.
  $$
  Here $\b, \g, C, c>0$ depend only on $\a$ and $M$.
\end{proposition}

\begin{proof}[Proof (sketch).]
Set $K:=(C t \log n)^{1/\a}$, and let $\tilde{G}$ be the matrix with entries
$\tilde{G}_{ij}=G_{ij} \mathbf{1}_{|G_{ij}|\le K}$.
Since $\E G_{ij}=0$, the bound on $\|G_{ij}\|_{\psi_\a}$ yields
$|\E \tilde{G}_{ij}| \le \exp (-cK^\a)$. Hence
$$
\| \E \tilde{G} \|  \le \| \E \tilde{G} \|_{\HS} \le n \exp (-cK^\a) \le n^{-1/2}.
$$
Then the event $\LL_{W,w}$ for the matrix $A = G-zI_n$ implies
the event $\LL_{W,w+1}$ for the matrix $\tilde{A} := G -\E \tilde{G} - z I_n$.
It remains to bound the probability of the latter event.
If the constant $C$ in the definition of $K$ is sufficiently large,
then with probability at least $1-n^{-t}$ we have $\tilde{G} = G$
and thus $\tilde{A} = \tilde{G} -\E \tilde{G} - z I_n$.
Conditioned on this likely event, the entries $\tilde{G} -\E \tilde{G}$ are independent,
bounded by $K$, have zero means and variances at least $1/2$.
Therefore, we can apply Proposition~\ref{prop: localization unlikely} for the matrix $\tilde{A}$
and thus bound the probability of $\LL_{W,w+1}$ for $\tilde{A}$, as required.
\end{proof}

\medskip

Corollary \ref{cor: psi_alpha} follows from Proposition \ref{prop: localization unlikely-general}
in the same way as Corollary \ref{cor: delocalization all z} followed from Proposition \ref{prop: localization unlikely}.
The only minor difference is that one would put a coarser bound the norm of $G$.
For example, one can use that $\|G\| \le \|G\|_\HS \le n \cdot \max_{i,j \le n}|G_{ij}| \le n \cdot Ms$ with probability
at least $1-2 n^2 \exp(-c s^\a)$, for any $s>0$.
  This, however, would only affect the bound on the covering number $N$ in Corollary \ref{cor: delocalization all z}, changing the estimate in this Corollary to
  \[
    \Pr{\LL_W} \le C (Ms)^2 n^6 \exp(-cl/\log n).
  \]
We omit the details.
\qed

\appendix

\section{Invertibility of random matrices}				\label{app}
%--------------

Our delocalization method relied on estimates of the smallest singular values of rectangular random matrices.
The method works well provided one has access to estimates that are polynomial in the dimension of the matrix
(which sometimes was of order $n$, and other times of order $l \sim \log^2 n$),
and provided the probability of having these estimates is, say, at least $1 - n^{-10}$.

In the recent years, significantly sharper bounds were proved than those required in our delocalization method,
see survey \cite{RV ICM}.
We chose to include weaker bounds in this appendix for two reasons. First, they hold
in somewhat more generality than those recorded in the literature, and also their proofs
are significantly simpler.

\begin{theorem}[Rectangular matrices]			\label{thm: rect}
  Let $N \ge n$, and let $A=D+G$ where
  $D$ is an arbitrary $N \times n$ fixed matrix and $G$ is an $N \times n$ random matrix with
  independent entries satisfying Assumption~\ref{ass}.
  Then
  \begin{equation}				\label{eq: snA}
  \Pr{ s_n(A) < c \sqrt{\frac{N-n}{n}} } \le 2 n \exp(-c(N-n)).
  \end{equation}
  Here $c=c(K)>0$.
\end{theorem}

\begin{proof}
Using the negative second moment identity (see \cite{TV ESD} Lemma A.4]), we have
\begin{equation}				\label{eq: sm0}
s_n(A)^{-2}
\le \sum_{i=1}^n s_i(A)^{-2}
= \sum_{i=1}^n d(A_i, E_i)^{-2}
\end{equation}
where $A_i=D_i+G_i$ denote the columns of $A$ and $E_i = \Span(A_j)_{j \le n, \, j \ne i}$.
For fixed $i$, note that
$d(A_i, E_i) = \|P_{E_i^\perp} A_i\|_2$.
Since $A_i$ is independent of $E_i$, we can apply the small ball probability
bound, Theorem~\ref{thm: concentration}(ii). Using that $\|P_{E_i^\perp}\|_\HS^2 = \dim(E_i^\perp) \ge N-n$
and $\|P_{E_i^\perp}\| = 1$, we obtain
$$
\Pr{ d(A, E_i) < c \sqrt{N-n} } \le 2 \exp(-c(N-n)).
$$

Union bound yields that with probability at least $1 - 2 n \exp(-c(N-n))$, we have
$d(A_i, E_i) \ge c \sqrt{N-n}$ for all $i \le n$.
Plugging this into \eqref{eq: sm0}, we conclude that with the same probability,
$s_n(A)^{-2} \le c^{-2} n / (N-n)$.
This completes the proof.
\end{proof}

\begin{corollary}[Intermediate singular values]			\label{thm: intermediate}
  Let $A=D+G$ where $D$ is an arbitrary $N \times M$ fixed matrix
  and $G$ is an $N \times M$ random matrix with
  independent entries satisfying Assumption~\ref{ass}.
  Then all singular values $s_n(A)$ for $1 \le n \le \min(N,M)$
  satisfy the estimate \eqref{eq: snA} with
  $c=c(K)>0$.
\end{corollary}

\begin{proof}
Recall that $s_n(A) \ge s_n(A_0)$ where $A_0$ is formed by the first $n$ columns of $A$.
The conclusion follows from Theorem~\ref{thm: rect} applied to $A_0$.
\end{proof}

\begin{theorem}[Products of random and deterministic matrices]				\label{thm: rectangular}
  Let $k,m,n \in \N$, $m \le \min(k,n)$.
  Let $P$ be a fixed $m \times n$ matrix such that $P P^\tran = I_m$,
  and $G$ be an $n \times k$ random matrix with independent entries that satisfy
  Assumption~\ref{ass}. Then
  $$
  \Pr{ s_m(PG) < c \, \frac{k-m}{k} } \le 2k \exp(-c(k-m)).
  $$
  Here $c=c(K)$.
\end{theorem}

Let us explain the idea of the proof of Theorem~\ref{thm: rectangular}.
We need a lower bound for
$$
\|(PG)^* x\|_2^2 = \sum_{i=1}^k \< PG_i, x\> ^2,
$$
where $G_i$ denote the columns of $G$. The bound has to be uniform over $x \in S^{m-1}$.
Let $m=(1-\d)k$ and set $m_0 = (1-\rho) m$ for a suitably chosen $\rho \ll \d$.

First, we claim that if
$x \in \Span(PG_i)_{i \le m_0} =: E$ then $\sum_{i=1}^{m_0} \< PG_i, x\> ^2 \gtrsim \|x\|_2^2$.
This is equivalent to controlling the smallest singular value
of the $m \times m_0$ random matrix with independent columns $PG_i$, $i=1,\ldots,m_0$.
Since $m \ge m_0$, this can be achieved with a minor variant of Theorem~\ref{thm: rect}.
The same argument works for general $x \in \C^m$
provided $x$ is not almost orthogonal onto $E$.

The vectors $x$ that lie near the subspace $E^\perp$, which has dimension
$m-m_0 = \rho m$, can be controlled by the remaining $k-m_0$ vectors $PG_i$,
since $k-m_0 \gg m-m_0$.
Indeed, this is equivalent to controlling the smallest singular value of a $(m-m_0) \times (k-m_0)$
random matrix whose columns are $Q G_i$, where $Q$ projects onto $E^\perp$.
This is a version of Theorem~\ref{thm: rectangular} for very fat matrices, and it
can be proved in a standard way by using $\e$-nets.

Now we proceed to the formal argument.

\begin{lemma}[Slightly fat matrices]					\label{lem: fat}
  Let $m_0 \le m$. Consider the $m \times m_0$ matrix $T_0$ formed by the first
  $m_0$ columns of matrix $T = PG$. Then
  $$
  \Pr{ s_{m_0}(T_0) < c \sqrt{\frac{m-m_0}{m_0}} } \le 2 m_0 \exp(-c(m-m_0)).
  $$
\end{lemma}

This is a minor variant of Theorem~\ref{thm: rect}; its proof is very similar and is omitted. \qed
%\begin{proof}
%Using the negative second moment identity, Lemma~\ref{lem: nsmi}, we have
%\begin{equation}				\label{eq: sm0}
%s_{m_0}(T_0)^{-2} \le \sum_{i=1}^{m_0} d(PG_i, E_i)^{-2}
%\end{equation}
%where $G_i$ denote the columns of $G$ and $E_i = \Span(PG_j)_{j \le m_0, \, j \ne i}$.
%Let us fix $i$ and note that
%$$
%d(PG_i, E_i) = \|P_{E_i^\perp} P G_i\|_2 =: \|Q G_i\|_2.
%$$
%Clearly, $\|Q\| \le 1$ and
%$\|Q\|_\HS = \|Q^\tran\|_\HS = \|P^\tran P_{E_i^\perp}\|_\HS = \|P_{E_i^\perp}\|_\HS
%\ge \sqrt{m-m_0}$. (Here we used that $P^\tran$ is an isometric embedding.)
%Since $G_i$ is independent of $Q$, we can apply the small ball probability
%bound of Corollary~\ref{cor: SBP}, and obtain
%$$
%\Pr{ d(PG_i, E_i) < c \sqrt{m-m_0} } \le 2 \exp(-c(m-m_0)).
%$$

%Union bound yields that with probability at least $1 - 2 m_0 \exp(-c(m-m_0))$, we have
%$d(PG_i, E_i) \ge c \sqrt{m-m_0}$ for all $i \le m_0$.
%Plugging this into \eqref{eq: sm0}, we conclude that with the same probability,
%$s_{m_0}(T_0)^{-2} \le c^{-2} m_0 / (m-m_0)$.
%This completes the proof.
%\end{proof}

\begin{lemma}[Very tall matrices]				\label{lem: tall}
  There exist $C = C(K)$, $c = c(K) > 0$ such that the following holds.
  Consider the same situation as in Theorem~\ref{thm: rectangular}, except that
  we assume that $k \ge C m$.
  Then
  $$
  \Pr{ s_m(PG) < c \sqrt{k} } \le \exp(-c k).
  $$
\end{lemma}

Lemma~\ref{lem: tall} is a minor variation of \cite[Theorem~5.39]{V RMT}
for $k \ge Cm$ independent sub-gaussian columns, and it can be proved in a similar way
(using a standard concentration and covering argument). \qed

\medskip

\begin{proof}[Proof of Theorem~\ref{thm: rectangular}]
Denote $T := PG$; our goal is to bound below the quantity
$$
s_m(T) = s_m(T^*) = \inf_{x \in S^{m-1}} \|T^* x\|_2^2.
%= \inf_{x \in S^{m-1}} \sum_{i=1}^k \< PG_i, x\> ^2.
$$
Let $\e, \rho \in (0,1/2)$ be parameters,
and set $m_0 = (1-\rho) m$.
We decompose
$$
T = [ T_0 \; \bar{T} ]
$$
where $T_0$ is the $m \times m_0$ matrix that consists of the first $m_0$ columns of $T$,
and $\bar{T}$ is the $(k-m_0) \times m$ matrix that consists of the last $k-m_0$ columns of $T$.
Let $x \in S^{m-1}$. Then
$$
\|T^*x\|_2^2 = \|T_0^* x\|_2^2 + \|\bar{T}^* x\|_2^2.
$$
Denote
$$
E = \im(T_0) = \Span(PG_i)_{i \le m_0}.
$$
Assume that $s_{m_0}(T_0) > 0$ (which will be seen to be a likely event),
so $\dim(E) = m_0$.

The argument now splits according to the position of $x$ relative to $E$.
Assume first that $\|P_E x\|_2 \ge \e$.
Since $\rank(T_0)=m_0$, using Lemma~\ref{lem: smallest sv}(i) we have
$$
\|T^* x\|_2 \ge \|T_0^* x\|_2 \ge s_{m_0}(T_0^*) \|P_E x\|_2 \ge s_{m_0}(T_0) \e.
$$
We will later apply Lemma~\ref{lem: fat} to bound $s_{m_0}(T_0)$ below.

Consider now the opposite case, where $\|P_E x\|_2 < \e$.  %Then $\dim(E^\perp) = m-m_0$.
There exists $y \in E^\perp$ such that $\|x-y\|_2 \le \e$, and in particular
$\|y\|_2 \ge \|x\|_2-\e \ge 1-\e > 1/2$. Thus
\begin{equation}				\label{eq: Tx below}
\|T^* x\|_2 \ge \|\bar{T}^* x\|_2 \ge \|\bar{T}^* y\|_2 - \|\bar{T}^*\| \e.
\end{equation}

We represent $\bar{T} = P \bar{G}$, where $\bar{G}$ is the $n \times (k-m_0)$ matrix
that contains the last $k-m_0$ columns of $G$.
Consider an $m \times (m-m_0)$ matrix $Q^*$ which is an isometric embedding
of $\ell_2^{m-m_0}$ into $\ell_2^m$,
and such that $\im(Q^*) = E^\perp$. Then there exists
$$
z \in \C^{m-m_0} \text{ such that } y = Q^* z, \quad \|z\|_2 = \|y\|_2 \ge 1/2.
$$
Therefore
$$
\|\bar{T}^* y\|_2 = \|\bar{G}^* P^* Q^* z\|_2.
$$
Since both $Q^* : \C^{m-m_0} \to \C^m$ and $P^* : \C^m \to \C^n$
are isometric embeddings, $R^* := P^* Q^* : \C^{m-m_0} \to \C^n$
is an isometric embedding, too. Thus $R$ is a $(m-m_0) \times n$
matrix which satisfies $RR^*=I_{m-m_0}$. Hence
$$
\|\bar{T}^* y\|_2 = \|\bar{B}^* z\|_2, \text{ where } \bar{B} := R\bar{G}
$$
is an $(m-m_0) \times (k-m_0)$ matrix.
Since $\|z\|_2 \ge 1/2$, we have
$\|\bar{T}^* y\|_2 \ge \frac{1}{2} s_{m-m_0}(B)$,
which together with \eqref{eq: Tx below} yields
$$
\|T^* x\|_2 \ge \frac{1}{2} s_{m-m_0}(\bar{B}) - \|\bar{T}\| \e.
$$
A bit later, we will use Lemma~\ref{lem: tall} to bound $s_{m-m_0}(\bar{B})$ below.

Putting the two cases together, we have shown that
\begin{equation}				\label{eq: smT}
s_m(T) \ge \min_{x \in S^{m-1}} \|T^* x\|_2
\ge \min \Big\{ s_{m_0}(T_0) \e, \, \frac{1}{2} s_{m-m_0}(\bar{B}) - \|\bar{T}\| \e \Big\}.
\end{equation}
It remains to estimate $s_{m_0}(T_0)$, $s_{m-m_0}(\bar{B})$ and $\|\bar{T}\|$.

Since $m_0 = (1-\rho)m$ and $\rho \in (0,1/2)$,
Lemma~\ref{lem: fat} yields that with probability at least $1- 2 m \exp(-c \rho m)$, we have
$$
s_{m_0}(T_0) \ge c \sqrt{\rho}.
$$

Next, we use Lemma~\ref{lem: tall} for the $(m-m_0) \times (k-m_0)$ matrix $\bar{B} = R\bar{G}$.
Let $\d \in (0,1)$ be such that $m=(1-\d)k$. Since $m_0=(1-\rho)m$, by choosing
$\rho = c_0 \d$
with a suitable $c_0>0$ we can achieve that $k-m_0 \ge C(m-m_0)$ to satisfy the
dimension requirement in Lemma~\ref{lem: tall}. Then, with probability
at least $1-2\exp(-c \d k)$ we have
$$
s_{m-m_0}(\bar{B}) \ge c \sqrt{\d k}.
$$

Further, by Theorem~\ref{thm: product norm}, with probability at least $1-2\exp(-k)$ we have
$$
\|\bar{T}\| \le \|T\| \le C \sqrt{k}.
$$

Putting all these estimates in \eqref{eq: smT}, we find that with
probability at least $1 - 2m \exp(-c \rho m) - 2\exp(-c \d k) - 2\exp(-k)$, one has
$$
s_m(T) \ge \min \Big\{c \sqrt{\rho} \, \e, \, \frac{1}{2} c \sqrt{\d k} - C \sqrt{k} \, \e \Big\}.
$$
Now we choose $\e = c_1 \sqrt{\d}$ with a suitable $c_1 > 0$,
and recall that we have chosen $\rho = c_0 \d$. We conclude that
$s_m(T) \ge c \min \{ \d, \, \sqrt{\d k} \} = c \d$.
Since $m = (1-\d) k$, the proof of Theorem~\ref{thm: rectangular} is complete.
\end{proof}

\end{document}